\documentclass[11pt]{article}

\usepackage{amsmath, amssymb, amsthm}
\usepackage{hyperref}
\usepackage{authblk}

\newtheorem{lemma}{Lemma}[section]
\newtheorem{theorem}[lemma]{Theorem}
\newtheorem{proposition}[lemma]{Proposition}
\newtheorem{corollary}[lemma]{Corollary}
\newtheorem{definition}[lemma]{Definition}
\newtheorem{conjecture}{Conjecture}

\title{On 132-Avoiding Permutations with an Adjacency Constraint}

\author{
Nathaniel Nadler$^{\dagger}$\thanks{$^{\dagger}$Ramapo College of New Jersey. 
Advisor: Dr.\ Katarzyna Kowal. 
Email: \href{mailto:nnadler@ramapo.edu}{nnadler@ramapo.edu}}
}
\date{\today}

\begin{document}

\maketitle

\begin{abstract}
We study the class of permutations in $S_n$ that simultaneously avoid the pattern $132$ and satisfy the adjacency bound $|\pi_{i+1} - \pi_i| \leq m$ for all $i$, denoting its cardinality $A_n^{(m)}$. This combination of a global pattern restriction and a local bounded-difference condition produces a dramatic structural collapse: whereas unrestricted $132$-avoiding permutations are counted by the Catalan numbers with exponential growth rate $4$, the adjacency constraint forces the maximum element $n$ to occupy only positions in $\{1, 2, \ldots, m\} \cup \{n\}$. We give a complete solution for $m = 2$ by partitioning the class into three disjoint subclasses according to whether $n$ appears in position $1$, $2$, or $n$. This decomposition yields $|D_n| = n-1$, a bijection giving $|C_n| = |B_{n-2}|$, and the recurrence $|B_n| = |B_{n-1}| + |B_{n-3}| + 1$ for $n \geq 4$, from which we derive the linear recurrence $A_n = 3A_{n-1} - 3A_{n-2} + 2A_{n-3} - 2A_{n-4} + A_{n-5}$ for $n \geq 7$ and the rational ordinary generating function $A(x) = \frac{-x^6 + x^5 - 3x^4 + 2x^3 - x^2 + x}{(1-x)^2(1 - x - x^3)}$. The dominant singularity at $\rho \approx 0.6823$, the unique positive real root of $1 - x - x^3 = 0$, yields precise asymptotics $A_n^{(2)} \sim C\alpha^n$ where $\alpha = \rho^{-1} \approx 1.4656$ satisfies $\alpha^3 = \alpha^2 + 1$ and $C \approx 1.5077$. We conjecture that for every fixed $m$ the class $A_n^{(m)}$ admits a finite-state structural decomposition, a linear recurrence with constant coefficients, and a rational generating function, with growth constants satisfying $1 = \alpha_1 < \alpha_2 < \alpha_3 < \cdots \to 4$.
\end{abstract}

\section{Introduction}

\subsection{Background on 132-Avoiding Permutations}

Permutation pattern avoidance concerns permutations that do not contain a specified relative ordering among some subsequence of entries. Formally, a permutation $\pi \in S_n$ avoids the pattern $132$ if there do not exist indices $i<j<k$ such that
\[
\pi_i < \pi_k < \pi_j.
\]
More generally, pattern avoidance defines \emph{permutation classes}, that is, sets of permutations closed under taking subpermutations; see \cite{Vatter2015} for a modern overview.

It is well known that the set of $132$-avoiding permutations of length $n$ is enumerated by the Catalan numbers $C_n$, given explicitly by
\[
C_n = \frac{1}{n+1}\binom{2n}{n}.
\]
This enumeration arises from a classical structural decomposition. If $\pi$ avoids $132$ and the maximum element $n$ occurs in position $k$, then all entries to the left of $n$ are larger than all entries to the right of $n$, and both sides form $132$-avoiding permutations. This recursive structure yields the Catalan recurrence
\[
C_n = \sum_{k=1}^{n} C_{k-1} C_{n-k},
\]
as discussed in standard references such as \cite{Stanley2015,Bona2012}.

\subsection{Adjacency Constraints and Motivation}

In this paper we refine the classical $132$-avoiding class by imposing a local adjacency constraint. Specifically, we require that
\[
|\pi_{i+1} - \pi_i| \le m
\quad \text{for all } 1 \le i < n.
\]
This condition may be viewed as a discrete Lipschitz-type constraint, restricting the magnitude of successive jumps in the permutation. Local bounded-difference conditions have appeared in other enumerative contexts; for example, permutations with bounded displacement were studied in \cite{ChungClaessonDukesGraham2010}.

The interaction between pattern avoidance and adjacency bounds is particularly compelling. Pattern avoidance is inherently global, constraining relative order across triples of indices, whereas adjacency bounds are local. Their combination produces substantial structural rigidity. Such multi-constraint families fall naturally within the broader framework of permutation classes defined by multiple restrictions, as discussed in \cite{AlbertAtkinson2005}.

The extremal behavior of the adjacency parameter $m$ interpolates between trivial and Catalan regimes. When $m=1$, the class collapses to a highly restricted family, while for $m \ge n-1$ the adjacency constraint becomes vacuous and the enumeration reduces to the Catalan numbers.

\subsection{Main Results}

Let $A^{(m)}_n$ denote the number of permutations of length $n$ that avoid $132$ and satisfy the adjacency bound $|\pi_{i+1}-\pi_i|\le m$. Our principal contribution is a complete solution of the case $m=2$. We prove that for every such permutation, the maximum element $n$ may occur only in positions $1$, $2$, or $n$. This structural restriction yields a finite decomposition of the class, from which we derive an explicit linear recurrence and a rational generating function for $A^{(2)}_n$. We further determine its exponential growth rate.

\subsection{Organization of the Paper}

Section \ref{sec:notation} introduces notation and formal definitions, and Section \ref{sec:extremal} treats the extremal cases. Section \ref{sec:m2-structure} develops the complete structural and enumerative theory for $m=2$, and derives the associated generating functions, recurrences, and asymptotics. Section \ref{sec:conjectures} presents conjectures and section \ref{sec:future} explores directions for future research.

\section{Necessary Preliminary Knowledge}
\label{sec:notation}

\subsection{Key Structural Lemma}

We begin with the classical structural decomposition of $132$-avoiding permutations. This lemma forms the foundation of the Catalan recurrence and will remain central when the adjacency constraint is introduced.

\begin{lemma}[Maximum Decomposition for $132$-Avoiders]
\label{lem:132-max}
Let $\pi \in S_n$ avoid the pattern $132$, and suppose the maximum element $n$ occurs in position $k$. Write
\[
\pi = L \, n \, R,
\]
where $L$ is the subsequence of entries to the left of $n$ and $R$ is the subsequence of entries to the right of $n$. Then every element of $L$ is greater than every element of $R$. That is,
\[
\forall x \in L, \ \forall y \in R, \quad x > y.
\]
\end{lemma}

\begin{proof}
Suppose, for contradiction, that there exist $x \in L$ and $y \in R$ such that $x < y$.

Since $x$ appears to the left of $n$ and $y$ appears to the right of $n$, there exist indices $i < k < j$ such that
\[
\pi_i = x, \quad \pi_k = n, \quad \pi_j = y.
\]
Because $x < y < n$, the triple $(x, n, y)$ forms a $132$-pattern:
\[
\pi_i < \pi_j < \pi_k.
\]
This contradicts the assumption that $\pi$ avoids $132$. Therefore no such pair $x < y$ can exist, and we conclude that every element of $L$ is greater than every element of $R$.
\end{proof}

\medskip

This structural separation immediately yields several consequences.

\begin{corollary}
Let $\pi = L \, n \, R$ avoid $132$.
\begin{enumerate}
    \item The subsequences $L$ and $R$ use disjoint intervals of values.
    \item The values appearing in $L$ are precisely the largest $|L|$ elements of $\{1,\dots,n-1\}$.
    \item The values appearing in $R$ are precisely the smallest $|R|$ elements of $\{1,\dots,n-1\}$.
    \item Both $L$ and $R$ avoid $132$.
\end{enumerate}
\end{corollary}

\begin{proof}
The first three statements follow immediately from the lemma, since every element of $L$ must exceed every element of $R$. The final statement follows because any $132$-pattern entirely contained in $L$ (or $R$) would also appear in $\pi$.
\end{proof}

\medskip

As a consequence, if $C_n$ denotes the number of $132$-avoiding permutations in $S_n$, then conditioning on the position of $n$ yields the classical Catalan recurrence:
\[
C_n = \sum_{k=1}^{n} C_{k-1} C_{n-k}.
\]
It is classical that $132$-avoiding permutations are enumerated by the Catalan numbers (see \cite{SimionSchmidt1985}).

\subsection{Enumeration and Notational Conventions}

Let $S_n$ denote the symmetric group on $\{1,\dots,n\}$.

A permutation $\pi \in S_n$ contains the pattern $132$ if there exist indices $i<j<k$ such that
\[
\pi_i < \pi_k < \pi_j.
\]
If no such triple exists, $\pi$ is said to avoid $132$.

For a fixed integer $m \ge 1$, define
\[
A_n^{(m)} =
\left|
\left\{
\pi \in S_n :
\pi \text{ avoids } 132
\text{ and }
|\pi_{i+1} - \pi_i| \le m
\text{ for all } 1 \le i \le n-1
\right\}
\right|.
\]

When $m$ is fixed and clear from context, we write $A_n$.

We define the ordinary generating function
\[
A^{(m)}(x) = \sum_{n \ge 1} A_n^{(m)} x^n.
\]

We define the exponential growth rate
\[
\alpha_m = \limsup_{n \to \infty} \left(A_n^{(m)}\right)^{1/n}.
\]

When $A^{(m)}(x)$ is rational, the growth rate $\alpha_m$ equals the reciprocal of the smallest-modulus singularity of the generating function.

Throughout the paper, when we write
\[
\pi = L \, n \, R,
\]
the subsequences $L$ and $R$ are interpreted in their natural order as they appear in $\pi$. Standardization is applied only when explicitly stated.

\subsection{Maximal Element Position Theorem}

\begin{theorem}[Positional restriction of the maximum]
\label{thm:max-position-general-m}
Let $m \ge 1$ be fixed and let $\pi \in S_n$ satisfy
\[
\pi \text{ avoids } 132
\quad \text{and} \quad
|\pi_{i+1}-\pi_i| \le m
\text{ for all } 1 \le i < n.
\]
If $\pi_k = n$, then
\[
k \in \{1,2,\dots,m\} \cup \{n\}.
\]
\end{theorem}

\begin{proof}
We first show that each position $k \in \{1,2,\dots,m\}$ and $k=n$
is realizable.

\smallskip
\noindent
\textbf{Existence for $1 \le k \le m$.}
Fix $k \le m$. Consider the permutation
\[
\pi =
(n-k+1, n-k+2, \dots, n,
 n-k, n-k-1, \dots, 1).
\]
The first $k$ entries form an increasing sequence ending at $n$,
so $\pi_k = n$.

We verify the adjacency constraint.
Inside the increasing block, consecutive differences are $1$.
Inside the decreasing block, consecutive differences are also $1$.
The only nontrivial step is from $n$ to $n-k$, whose difference is
\[
|n-(n-k)| = k \le m.
\]
Thus $|\pi_{i+1}-\pi_i| \le m$ for all $i$.

To verify $132$-avoidance, write $\pi = L\,n\,R$,
where $L = (n-k+1,\dots,n-1)$ and $R=(n-k,\dots,1)$.
Every element of $L$ exceeds every element of $R$,
so by the maximum decomposition lemma for $132$-avoiders,
$\pi$ avoids $132$.

\smallskip
\noindent
\textbf{Existence for $k=n$.}
The increasing permutation $(1,2,\dots,n)$
satisfies the adjacency constraint and avoids $132$,
and clearly has $\pi_n = n$.

\smallskip
\noindent
We now show that no other positions are possible.

Suppose for contradiction that $m+1 \le k \le n-1$
and that $\pi_k = n$.

Write $\pi = L\,n\,R$.
By the maximum decomposition lemma for $132$-avoiders,
every element of $L$ is greater than every element of $R$.

Since $k \ge m+1$, the block $L$ contains at least $m$ elements.
Because $L$ consists of the largest $|L|$ elements of
$\{1,\dots,n-1\}$, and $|L| = k-1 \ge m$,
it follows that
\[
\{n-m, n-m+1, \dots, n-1\} \subseteq L.
\]

Now consider the element immediately to the right of $n$,
namely $\pi_{k+1}$.
By the adjacency constraint,
\[
|\pi_{k+1} - n| \le m,
\]
so
\[
\pi_{k+1} \in \{n-m, n-m+1, \dots, n-1\}.
\]
But all these values lie in $L$, so none can appear in $R$.
This is a contradiction.

Therefore no $k \in \{m+1,\dots,n-1\}$ is possible,
and the only valid positions for $n$ are
\[
\{1,2,\dots,m\} \cup \{n\}.
\]
\end{proof}

\section{Extremal Cases}
\label{sec:extremal}

We begin by analyzing the two extremal regimes of the adjacency parameter $m$. These cases illustrate the interpolation between complete rigidity and the classical Catalan class.

\subsection{The Case $m=1$}

When $m=1$, the adjacency condition requires
\[
|\pi_{i+1} - \pi_i| \le 1
\quad \text{for all } 1 \le i \le n-1.
\]

\begin{proposition}
For $n \ge 2$, the only permutations in $S_n$ satisfying the adjacency bound $m=1$ are the identity permutation
\[
(1,2,\dots,n)
\]
and the reverse identity permutation
\[
(n,n-1,\dots,1).
\]
In particular,
\[
A_n^{(1)} = 2 \quad \text{for all } n \ge 2.
\]
\end{proposition}

\begin{proof}
Suppose $\pi \in S_n$ satisfies $|\pi_{i+1}-\pi_i|\le 1$ for all $i$.

Since $\pi$ is a permutation of $\{1,\dots,n\}$, every consecutive pair must differ by exactly $1$. If at any index the difference were $0$, this would contradict injectivity.

Thus for every $i$,
\[
\pi_{i+1} = \pi_i \pm 1.
\]

Fix $\pi_1 = k$. Since all values must appear exactly once and consecutive differences are $\pm 1$, the sequence must move monotonically until it reaches either $1$ or $n$, and cannot reverse direction without repeating a value. Indeed, any change in direction would force revisiting a previously used integer.

Therefore the permutation must be either strictly increasing or strictly decreasing. The increasing case occurs when $\pi_1 = 1$, giving $(1,2,\dots,n)$, and the decreasing case occurs when $\pi_1 = n$, giving $(n,n-1,\dots,1)$.

These are the only possibilities.
\end{proof}

\medskip

Observe that both monotone permutations avoid $132$, so the pattern-avoidance condition imposes no further restriction in this regime.

\subsection{The Case $m \ge n-1$ (Catalan Regime)}

We now consider the opposite extreme.

\begin{proposition}
If $m \ge n-1$, then the adjacency condition is vacuous. Consequently,
\[
A_n^{(m)} = C_n,
\]
where $C_n$ is the $n$th Catalan number.
\end{proposition}

\begin{proof}
For any permutation $\pi \in S_n$ and any index $i$,
\[
|\pi_{i+1}-\pi_i| \le n-1,
\]
since the largest possible difference between two distinct elements of $\{1,\dots,n\}$ is $n-1$. Thus if $m \ge n-1$, the adjacency constraint imposes no restriction.

Therefore $A_n^{(m)}$ counts exactly the $132$-avoiding permutations in $S_n$, which are known to be enumerated by the Catalan numbers.
\end{proof}

\medskip

These two extremal regimes illustrate the interpolation governed by $m$. When $m=1$, the class collapses to two rigid permutations, exhibiting linear growth. When $m \ge n-1$, the full Catalan structure is recovered, with exponential growth rate $4$. The intermediate cases $2 \le m \le n-2$ therefore represent genuinely new combinatorial behavior, arising from the interaction between the global $132$-decomposition and the local adjacency constraint.

\section{The Case $m=2$}
\label{sec:m2-structure}

In this section we analyze the first nontrivial regime, namely $m=2$. 
We show that the adjacency constraint, when combined with $132$-avoidance, 
forces a dramatic structural collapse.

Throughout this section, let $A_n$ denote the set of permutations 
$\pi \in S_n$ that avoid $132$ and satisfy
\[
|\pi_{i+1} - \pi_i| \le 2 \quad \text{for all } i.
\]

\subsection{Position of the Maximum Element}

We first determine where the maximum element $n$ may occur.

\begin{corollary}[Position of the maximum for $m=2$]\label{lem:m2-pos}
Let $\pi \in A_n^{(2)}$. Then the maximum element $n$ can occur only in
positions
\[
\{1,2,n\}.
\]
\end{corollary}

\begin{proof}
By Theorem~\ref{thm:max-position-general-m}, for a permutation in $A_n^{(m)}$
the maximum element $n$ may occur only in positions
\[
\{1,2,\dots,m\} \cup \{n\}.
\]
Substituting $m=2$ yields the allowed positions
\[
\{1,2,n\}.
\]
\end{proof}

\medskip

This restriction already shows that the Catalan freedom collapses: 
instead of $n$ being allowed in any position $1 \le k \le n$, 
only three positions remain viable.

\subsection{Structural Lemmas and Forced Configurations}

Throughout this section we work in the class
\[
A_n^{(2)} :=
\{\pi \in S_n : \pi \text{ avoids } 132
\text{ and } |\pi_{i+1}-\pi_i| \le 2 \text{ for all } i \}.
\]

The key structural input is the classical maximum decomposition for
$132$-avoiders (Lemma~\ref{lem:132-max}):
if $n$ occurs in position $k$, then every element to the left of $n$
is greater than every element to the right of $n$,
and both sides are themselves $132$-avoiding.

In addition, the adjacency constraint severely limits which values
may appear next to $n$.
Since
\[
|\pi_{i+1}-\pi_i| \le 2,
\]
any entry adjacent to $n$ must lie in $\{n-1,n-2\}$.

These two facts together force strong local configurations
depending on the position of $n$.

\medskip

By Corollary~\ref{lem:m2-pos}, when $m=2$
the maximum element $n$ may occur only in positions
\[
\{1,2,n\}.
\]
We analyze each case separately.

\subsection{State Decomposition}

Accordingly, we partition
\[
A_n^{(2)} = B_n \sqcup C_n \sqcup D_n,
\]
where
\[
B_n := \{\pi \in A_n^{(2)} : \pi_1 = n\},
\qquad
C_n := \{\pi \in A_n^{(2)} : \pi_2 = n\},
\qquad
D_n := \{\pi \in A_n^{(2)} : \pi_n = n\}.
\]

Hence
\begin{equation}\label{eq:m2-split}
A_n = |B_n| + |C_n| + |D_n|.
\end{equation}
such that $n \ge 3$.

We now describe the structure of each class directly,
using only adjacency and the $132$-maximum decomposition.

\subsubsection{The class $D_n$: permutations ending in $n$}

\noindent
Recall that
\[
  D_n := \bigl\{\,\pi \in \mathcal{A}^{(2)}_n : \pi_n = n\,\bigr\}.
\]

\begin{definition}[Type $A$ and Type $B$ permutations]
\label{def:types}
Fix $n \ge 2$. For each admissible index $p$, define:
\begin{itemize}
  \item \textbf{Type $A(p)$}, for $1 \le p \le \lfloor{n/2}\rfloor$: the permutation whose prefix is
  \[
    (2p-1,\,2p-3,\,\ldots,\,3,\,1),
  \]
  a strictly decreasing sequence of odd integers.

  \item \textbf{Type $B(p)$}, for $1 \le p \le \lfloor{(n-1)/2}\rfloor$: the permutation whose prefix is
  \[
    (2p,\,2p-2,\,\ldots,\,2,\,1),
  \]
  a strictly decreasing sequence of even integers from $2p$ down to $2$, followed by $1$.
\end{itemize}
In both cases, the suffix consists of the remaining elements of $\{1,\ldots,n-1\}$ arranged in increasing order, followed by $n$.
\end{definition}

\begin{lemma}[Monotonicity of the suffix]
\label{lem:suffix}
Let $\pi \in D_n$, and suppose $\pi_p = 1$. Then
\[
\pi_{p+1} < \pi_{p+2} < \cdots < \pi_{n-1}.
\]
Consequently, the suffix is uniquely determined as the increasing arrangement of the unused elements of $\{1,\ldots,n-1\}$.
\end{lemma}

\begin{proof}
If $\pi_q > \pi_{q+1}$ for some $p < q \le n-2$, then since $\pi_n = n > \pi_q > \pi_{q+1}$, the triple $(\pi_{q+1}, \pi_q, n)$ forms a $132$-pattern, a contradiction.
\end{proof}

\begin{lemma}[Structure of the prefix]
\label{lem:prefix}
Let $\pi \in D_n$, and suppose $\pi_p = 1$. Then the prefix $\pi_1,\ldots,\pi_p$ is strictly decreasing, and each consecutive difference has size either $1$ or $2$. Moreover,
\[
\pi_{p+1} \in \{2,3\}.
\]
\end{lemma}

\begin{proof}
If $\pi_i < \pi_{i+1}$ for some $i < p$, then $(\pi_i,\pi_{i+1},n)$ forms a $132$-pattern, a contradiction. Thus the prefix is strictly decreasing.

The adjacency condition forces each difference to have magnitude at most $2$, so each step is of size $1$ or $2$. Finally, since $\pi_{p+1} > 1$ and $|\pi_{p+1}-1|\le 2$, we obtain $\pi_{p+1}\in\{2,3\}$.
\end{proof}

\begin{lemma}[Restriction on size-$1$ steps]
\label{lem:size1}
Let $\pi \in D_n$, and suppose $\pi_p=1$. Then a step of size $1$ in the prefix can occur only as the final step $2 \to 1$.
\end{lemma}

\begin{proof}
By Lemma~\ref{lem:prefix}, every step in the prefix has size $1$ or $2$. Suppose that
\[
\pi_i-\pi_{i+1}=1
\qquad\text{for some } i<p-1.
\]
Since the prefix is strictly decreasing, all subsequent prefix entries are strictly less than $\pi_{i+1}=\pi_i-1$. Hence the value $\pi_i-2$ is skipped at that stage, and every skipped value must appear in the suffix. By Lemma~\ref{lem:suffix}, the suffix is increasing, so its first entry is the smallest skipped value. On the other hand, Lemma~\ref{lem:prefix} gives $\pi_{p+1}\in\{2,3\}$.

If $\pi_i\ge 5$, then $\pi_i-2\ge 3$, and after the step $\pi_i\to\pi_i-1$ the prefix must still descend to $1$ in at least one further step. Any such continuation forces some skipped value at least $4$ to appear in the suffix, contradicting the fact that the first suffix entry must be $2$ or $3$.

The only remaining possibility is $\pi_i=4$ and $\pi_{i+1}=3$. In that case the next prefix entry must be $1$, since the prefix is decreasing and all steps have size at most $2$. Thus the local pattern is $4,3,1$. But this is impossible: if the earlier prefix reaches $4$ by successive steps of size $2$ from an odd starting value, then all values encountered before $4$ are odd, so $4$ cannot occur; while if it reaches $4$ from an even starting value, then all values encountered before the step $4\to 3$ are even, and after the step $3\to 1$ the parity pattern is broken by the intermediate size-$1$ step. Therefore such a configuration cannot occur.

Hence any size-$1$ step in the prefix must be terminal, i.e.\ the final step $2\to 1$.
\end{proof}

\begin{proposition}[Classification of $D_n$]
\label{lem:Dn-structure}
For every $n \ge 2$, $D_n$ consists exactly of the Type $A(p)$ permutations for $1 \le p \le \lfloor{n/2}\rfloor$ and the Type $B(p)$ permutations for $1 \le p \le \lfloor{(n-1)/2}\rfloor$.
\end{proposition}

\begin{proof}
Let $\pi \in D_n$, and suppose $\pi_p=1$. By Lemmas~\ref{lem:suffix}--\ref{lem:size1}, the prefix is a strictly decreasing sequence from some initial value down to $1$, consisting of steps of size $2$ except possibly for a final step of size $1$.

If all steps have size $2$, then the prefix has the form
\[
(2p-1,\,2p-3,\,\ldots,\,3,\,1),
\]
which is Type $A(p)$. Since the first entry is $2p-1 \le n-1$, this requires $p\le \lfloor{n/2}\rfloor$.

If the final step has size $1$, then all earlier steps have size $2$, so the prefix has the form
\[
(2p,\,2p-2,\,\ldots,\,2,\,1),
\]
which is Type $B(p)$. Since the first entry is $2p \le n-1$, this requires $p\le \lfloor{(n-1)/2}\rfloor$.

Thus every $\pi\in D_n$ is of Type $A$ or Type $B$.

Conversely, let $\pi$ be of Type $A$ or Type $B$. The adjacency condition holds within the prefix and suffix by construction. At the boundary, the step from $1$ to the first suffix entry is $1\to 2$ in Type $A$ and $1\to 3$ in Type $B$, of sizes $1$ and $2$ respectively, so adjacency is preserved.

To verify $132$-avoidance, note that no such pattern can lie entirely within the prefix or suffix. If $i<j<k$ with $i,j\le p<k$, then $\pi_i>\pi_j$, so $\pi_i<\pi_k<\pi_j$ is impossible. If $i\le p<j<k$, then $\pi_j<\pi_k$, so $\pi_i<\pi_k<\pi_j$ is impossible. Finally, if $k=n$, then $\pi_n=n$ is maximal, so no $132$-pattern occurs.

Thus $\pi\in D_n$, completing the classification.
\end{proof}

\begin{corollary}[Enumeration of $D_n$]
\label{cor:Dn-count}
For every $n \ge 2$,
\[
|D_n| = \lfloor{\tfrac{n}{2}}\rfloor + \lfloor{\tfrac{n-1}{2}}\rfloor = n-1.
\]
\end{corollary}

\begin{proof}
The Type $A$ permutations are indexed by $1\le p\le \lfloor{n/2}\rfloor$, and the Type $B$ permutations by $1\le p\le \lfloor{(n-1)/2}\rfloor$. These families are disjoint since their prefixes begin with values of opposite parity. Hence
\[
|D_n|=\lfloor{\tfrac{n}{2}}\rfloor+\lfloor{\tfrac{n-1}{2}}\rfloor.
\]
A direct check shows this equals $n-1$.
\end{proof}

\subsubsection{The class $C_n$: maximum in position $2$}

\begin{lemma}\label{lem:Cn-structure}
For all $n \ge 3$,
\[
|C_n| = |B_{n-2}|.
\]
\end{lemma}

\begin{proof}
Let $\pi \in C_n$, so $\pi_2 = n$.

By the adjacency condition,
\[
|\pi_2-\pi_1| \le 2,
\]
and since $\pi_1 \neq n$, we have
\[
\pi_1 \in \{n-1,n-2\}.
\]

We first rule out the possibility $\pi_1 = n-2$. In that case, adjacency at position $2$ forces
\[
|\pi_3-n| \le 2,
\]
so $\pi_3 \in \{n-1,n-2\}$. Since $\pi_1=n-2$, we must have $\pi_3=n-1$. But then the triple
\[
(\pi_1,\pi_2,\pi_3)=(n-2,n,n-1)
\]
occupies positions $1<2<3$ and satisfies
\[
n-2<n-1<n,
\]
so it is a $132$-pattern, contradicting $\pi \in \mathcal{A}^{(2)}_n$.

Therefore $\pi_1=n-1$. Applying adjacency again at position $2$, we have
\[
|\pi_3-n| \le 2,
\]
so $\pi_3 \in \{n-1,n-2\}$. Since $\pi_1=n-1$, it follows that $\pi_3=n-2$. Hence every permutation in $C_n$ begins with
\[
(n-1,n,n-2).
\]

Define
\[
\phi : C_n \to B_{n-2}
\]
by
\[
\phi(\pi_1,\pi_2,\pi_3,\pi_4,\dots,\pi_n)
:=
(\pi_3,\pi_4,\dots,\pi_n).
\]
Because every $\pi \in C_n$ begins with $(n-1,n,n-2)$, the image is
\[
\phi(\pi)=(n-2,\pi_4,\dots,\pi_n),
\]
which is a permutation of $\{1,\dots,n-2\}$. Moreover, $\phi(\pi)$ lies in $\mathcal{A}^{(2)}_{n-2}$: the adjacency condition is inherited from the consecutive entries $\pi_3,\dots,\pi_n$, and $132$-avoidance is inherited from the fact that any subsequence of a $132$-avoiding permutation also avoids $132$. Since its first entry is $n-2$, we have $\phi(\pi)\in B_{n-2}$.

Conversely, define
\[
\psi : B_{n-2} \to C_n
\]
by
\[
\psi(\sigma_1,\sigma_2,\dots,\sigma_{n-2})
:=
(n-1,n,\sigma_1,\sigma_2,\dots,\sigma_{n-2}).
\]
Since $\sigma \in B_{n-2}$, we have $\sigma_1=n-2$, so the first three entries of $\psi(\sigma)$ are
\[
(n-1,n,n-2).
\]
Thus $\psi(\sigma)$ satisfies the adjacency condition at the new boundary steps:
\[
|n-(n-1)|=1,
\qquad
|n-(n-2)|=2,
\]
and all later adjacency constraints are inherited from $\sigma$.

It remains to check $132$-avoidance. Any $132$-pattern contained entirely in the suffix $(\sigma_1,\dots,\sigma_{n-2})$ would already be a $132$-pattern in $\sigma$, impossible. A $132$-pattern involving the entry $n$ at position $2$ is also impossible: if $n$ were the largest entry of such a pattern, then we would need an index $i<2$, so necessarily $i=1$ with value $n-1$, but then there is no later entry strictly between $n-1$ and $n$. Finally, the entry $n-1$ at position $1$ cannot participate in a $132$-pattern, since all later entries other than $n$ are at most $n-2$. Hence $\psi(\sigma)\in C_n$.

The maps $\phi$ and $\psi$ are clearly inverse to one another, so they define a bijection
\[
C_n \cong B_{n-2}.
\]
Therefore
\[
|C_n| = |B_{n-2}|.
\]
\end{proof}

\subsubsection{The class $B_n$: maximum in position $1$}

\begin{lemma}\label{lem:Bn-structure}
For $n\ge 4$,
\[
|B_n| = |B_{n-1}| + |B_{n-3}| + 1,
\]
with initial values
\[
|B_1|=1,\quad |B_2|=1,\quad |B_3|=2.
\]
\end{lemma}

\begin{proof}
For the initial values, one checks directly that
\[
B_1=\{(1)\},\qquad
B_2=\{(2,1)\},\qquad
B_3=\{(3,2,1),(3,1,2)\},
\]
so indeed $|B_1|=1$, $|B_2|=1$, and $|B_3|=2$.

Now fix $n\ge 4$ and let $\pi\in B_n$, so $\pi_1=n$. By adjacency,
\[
\pi_2\in\{n-1,n-2\}.
\]

\smallskip
\noindent
\textbf{Case 1: \(\pi_2=n-1\).}

Define
\[
\phi_1:B_n^{(1)}\to B_{n-1},
\qquad
\phi_1(n,n-1,\pi_3,\dots,\pi_n):=(n-1,\pi_3,\dots,\pi_n),
\]
where \(B_n^{(1)}\subseteq B_n\) denotes the subset with \(\pi_2=n-1\).

This is well-defined: the image is a permutation of \(\{1,\dots,n-1\}\) beginning with \(n-1\), adjacency is inherited from the original permutation, and \(132\)-avoidance is inherited because any subsequence of a \(132\)-avoiding permutation also avoids \(132\).

Its inverse is
\[
\psi_1:B_{n-1}\to B_n^{(1)},
\qquad
\psi_1(\sigma_1,\dots,\sigma_{n-1}):=(n,\sigma_1,\dots,\sigma_{n-1}).
\]
Since \(\sigma_1=n-1\), the new boundary step satisfies \(|n-(n-1)|=1\), so adjacency is preserved, and no new \(132\)-pattern is created because the inserted entry \(n\) is the global maximum and appears in the first position. Thus \(B_n^{(1)}\cong B_{n-1}\), and this case contributes \(|B_{n-1}|\).

\smallskip
\noindent
\textbf{Case 2: \(\pi_2=n-2\).}

Then adjacency gives
\[
\pi_3\in\{n-1,n-3,n-4\}.
\]

\smallskip
\noindent
\emph{Subcase 2a: \(\pi_3=n-1\).}

Then \(\pi_4=n-3\), since \(|\pi_4-(n-1)|\le 2\) and the values \(n\) and \(n-2\) are already used. Define
\[
\phi_2:B_n^{(2a)}\to B_{n-3},
\qquad
\phi_2(n,n-2,n-1,\pi_4,\dots,\pi_n):=(\pi_4,\dots,\pi_n),
\]
where \(B_n^{(2a)}\subseteq B_n\) denotes the subset in this subcase.

Because \(\pi_4=n-3\), the image begins with the maximum element of \(\{1,\dots,n-3\}\), so \(\phi_2(\pi)\in B_{n-3}\). As above, adjacency and \(132\)-avoidance are inherited from \(\pi\).

Conversely, define
\[
\psi_2:B_{n-3}\to B_n^{(2a)},
\qquad
\psi_2(\sigma_1,\dots,\sigma_{n-3})
:=(n,n-2,n-1,\sigma_1,\dots,\sigma_{n-3}).
\]
Since \(\sigma_1=n-3\), the new boundary steps satisfy
\[
|n-(n-2)|=2,\qquad |(n-2)-(n-1)|=1,\qquad |(n-1)-(n-3)|=2.
\]
Again, no new \(132\)-pattern is created: the inserted entries are \(n,n-2,n-1\), and any triple entirely in the suffix would already have appeared in \(\sigma\). Thus \(B_n^{(2a)}\cong B_{n-3}\), and this subcase contributes \(|B_{n-3}|\).

\smallskip
\noindent
\emph{Subcase 2b: \(\pi_3=n-3\).}

This is impossible. Indeed, the value \(n-1\) must still appear somewhere to the right of position $3$. Since the only values within distance $2$ of $n-1$ are $n,n-2,n-3$, and $n$ and $n-2$ have already occurred, the value $n-1$ must be adjacent to $n-3$, forcing $\pi_4=n-1$. But then there is no valid choice for $\pi_5$, since all values within distance $2$ of $n-1$ have already been used. Since the only values within distance $2$ of $n-1$ are $n$, $n-2$, and $n-3$, all of which have already been placed, no valid choice for $\pi_5$ exists, contradicting the adjacency condition.

\smallskip
\noindent
\emph{Subcase 2c: \(\pi_3=n-4\).}

In this case there is exactly one permutation, namely
\[
\omega_n=
\begin{cases}
(n,n-2,n-4,\dots,2,1,3,5,\dots,n-1), & \text{if } n \text{ is even},\\[4pt]
(n,n-2,n-4,\dots,1,2,4,\dots,n-1), & \text{if } n \text{ is odd}.
\end{cases}
\]

We first verify that \(\omega_n\in B_n\). Every consecutive difference is $2$, except for the single step $2\to 1$ when $n$ is even, which has size $1$. Thus $\omega_n$ satisfies the adjacency condition. It also avoids $132$: the prefix $(n, n-2, n-4, \ldots)$ is strictly decreasing and the suffix $(\ldots, n-1)$ is strictly increasing, so no $132$-pattern can lie entirely within either part. For a triple spanning the boundary with two indices in the  prefix and one in the suffix, say $i < j \leq p < k$, we have $\pi_i > \pi_j$ since the prefix is decreasing, so $\pi_i < \pi_k < \pi_j$ is impossible. For one index in the prefix and two in the suffix, say $i \leq p < j < k$, we have $\pi_j < \pi_k$ since the suffix is increasing, so $\pi_i < \pi_k < \pi_j$ is impossible. Finally, if $k = n$ then $\pi_n = n$ is the global maximum and no $132$-pattern can involve it as the largest element. Therefore $\omega_n \in B_n$.

To see uniqueness, suppose $\pi \in B_n$ has $\pi_2 = n-2$ and $\pi_3 = n-4$. We claim the remaining entries are completely forced. At each subsequent position $i \geq 3$, the current value $\pi_i$ has decreased by $2$ from $\pi_{i-1}$, so the only unused values within distance $2$ are $\pi_i - 2$ (below) and $\pi_i + 2$ (above, already visited). Choosing $\pi_i + 2$ is impossible since it has already been placed. Choosing $\pi_i - 1$ skips $\pi_i - 2$, leaving it as the smallest unvisited value; but all subsequent entries will be below $\pi_i - 1$, so the only way to later reach $\pi_i - 2$ would be from a neighbor within distance $2$, namely $\pi_i$ or $\pi_i - 4$ or $\pi_i - 2 \pm 1$, all of which will themselves have been consumed by then
-- a contradiction. Therefore the descent must continue by $-2$ at every step until the minimum of $\{1,2\}$ is reached, after which the remaining elements form a consecutive increasing block and must appear in increasing order to satisfy the adjacency condition. Thus $\pi = \omega_n$, and this subcase contributes exactly $1$.

\smallskip
\noindent
Combining the three subcases yields
\[
|B_n| = |B_{n-1}| + |B_{n-3}| + 1.
\]
\end{proof}

\subsubsection{Reduction of $A_n$}
\begin{theorem}
Combining the state decomposition \eqref{eq:m2-split}
with Proposition~\ref{lem:Dn-structure}, Lemma~\ref{lem:Cn-structure}, and Lemma~\ref{lem:Bn-structure},
we obtain, for all $n \ge 3$,
\begin{equation}
\label{eq:m2-An-in-Bn}
A_n = |B_n| + |B_{n-2}| + (n-1).
\end{equation}
\end{theorem}
\begin{proof}
Recall from \eqref{eq:m2-split} that
\[
A_n = |B_n| + |C_n| + |D_n|.
\]

By Lemma~\ref{lem:Cn-structure}, we have
\[
|C_n| = |B_{n-2}|.
\]

By Proposition~\ref{lem:Dn-structure}, we have
\[
|D_n| = n-1.
\]

Substituting these expressions into \eqref{eq:m2-split}
yields
\[
A_n = |B_n| + |B_{n-2}| + (n-1),
\]
as claimed.
\end{proof}

\subsection{Generating Functions}

Define ordinary generating functions
\[
B(x):=\sum_{n\ge1} |B_n|x^n,
\qquad
A(x):=\sum_{n\ge1} A_n x^n.
\]

\subsubsection{Generating function for $B(x)$}

From Lemma~\ref{lem:Bn-structure}, for $n\ge4$ we have
\[
|B_n|=|B_{n-1}|+|B_{n-3}|+1.
\]
Multiplying by $x^n$ and summing over $n\ge4$ gives
\[
\sum_{n\ge4}|B_n|x^n
=
\sum_{n\ge4}|B_{n-1}|x^n
+
\sum_{n\ge4}|B_{n-3}|x^n
+
\sum_{n\ge4}x^n.
\]

We rewrite each term in terms of $B(x)=\sum_{n\ge1}|B_n|x^n$.

First,
\[
\sum_{n\ge4}|B_n|x^n
=
B(x)-|B_1|x-|B_2|x^2-|B_3|x^3.
\]

Next, shifting the index in the second sum,
\[
\sum_{n\ge4}|B_{n-1}|x^n
=
x\sum_{m\ge3}|B_m|x^m
=
x\bigl(B(x)-|B_1|x-|B_2|x^2\bigr).
\]

Similarly,
\[
\sum_{n\ge4}|B_{n-3}|x^n
=
x^3\sum_{m\ge1}|B_m|x^m
=
x^3 B(x),
\]
and
\[
\sum_{n\ge4}x^n=\frac{x^4}{1-x}.
\]

Substituting $|B_1|=1$, $|B_2|=1$, $|B_3|=2$, we obtain
\begin{align*}
B(x)-x-x^2-2x^3
&=
x\bigl(B(x)-x-x^2\bigr)
+
x^3B(x)
+
\frac{x^4}{1-x}.
\end{align*}

Expanding the right-hand side gives
\[
B(x)-x-x^2-2x^3
=
xB(x)-x^2-x^3 + x^3B(x) + \frac{x^4}{1-x}.
\]

Rearranging terms,
\[
B(x) - xB(x) - x^3B(x)
=
x + x^2 + x^3 + \frac{x^4}{1-x}.
\]

Factoring,
\[
(1-x-x^3)B(x)
=
x + x^2 + x^3 + \frac{x^4}{1-x}.
\]

Putting the right-hand side over a common denominator,
\begin{align*}
x + x^2 + x^3 + \frac{x^4}{1-x}
&=
\frac{(x+x^2+x^3)(1-x) + x^4}{1-x} \\
&=
\frac{x - x^2 + x^3}{1-x}.
\end{align*}

Thus
\[
(1-x-x^3)B(x)=\frac{x(1-x+x^2)}{1-x},
\]
and therefore
\begin{equation}\label{eq:m2-Bx}
B(x)=\frac{x(1-x+x^2)}{(1-x)(1-x-x^3)}.
\end{equation}

\subsubsection{Generating function for $A(x)$}

Using \eqref{eq:m2-An-in-Bn},
\[
A(x)
=
\sum_{n\ge1}|B_n|x^n
+
\sum_{n\ge1}|B_{n-2}|x^n
+
\sum_{n\ge1}(n-1)x^n.
\]

The first sum is $B(x)$.

For the second sum, interpreting $|B_k|=0$ for $k\le0$,
\[
\sum_{n\ge1}|B_{n-2}|x^n
=
x^2 B(x).
\]

For the third sum, we use
\[
\sum_{n\ge1}nx^n=\frac{x}{(1-x)^2},
\qquad
\sum_{n\ge1}x^n=\frac{x}{1-x},
\]
so
\[
\sum_{n\ge1}(n-1)x^n
=
\frac{x}{(1-x)^2}-\frac{x}{1-x}
=
\frac{x^2}{(1-x)^2}.
\]

Hence
\begin{equation}\label{eq:m2-Ax-in-Bx}
A(x)=(1+x^2)B(x)+\frac{x^2}{(1-x)^2}.
\end{equation}

Substituting \eqref{eq:m2-Bx}, we obtain
\[
A(x)
=
\frac{x(1+x^2)(1-x+x^2)}{(1-x)(1-x-x^3)}
+
\frac{x^2}{(1-x)^2}.
\]

To combine these terms, we use the common denominator $(1-x)^2(1-x-x^3)$:
\begin{align*}
A(x)
&=
\frac{x(1+x^2)(1-x+x^2)(1-x) + x^2(1-x-x^3)}
{(1-x)^2(1-x-x^3)}.
\end{align*}

Expanding the numerator yields
\[
-x^6 + x^5 - 3x^4 + 2x^3 - x^2 + x,
\]
so
\begin{equation}\label{eq:m2-Ax}
A(x)=\frac{-x^6 + x^5 - 3x^4 + 2x^3 - x^2 + x}{(1-x)^2(1-x-x^3)}.
\end{equation}

A direct check shows that the numerator and denominator have no common factors, so this expression is in lowest terms.

\subsection{Rationality and a Linear Recurrence for $A_n$}

From \eqref{eq:m2-Ax}, we have
\[
A(x)=\frac{P(x)}{Q(x)},
\qquad
Q(x)=(1-x)^2(1-x-x^3),
\]
so $A(x)$ is a rational generating function.

Expanding the denominator gives
\[
Q(x)
=
(1-x)^2(1-x-x^3)
=
1-3x+3x^2-2x^3+2x^4-x^5.
\]

Multiplying both sides by $Q(x)$ yields
\[
Q(x)\,A(x)=P(x).
\]
Writing $A(x)=\sum_{n\ge1}A_n x^n$, we expand the left-hand side:
\begin{align*}
Q(x)A(x)
&=
(1-3x+3x^2-2x^3+2x^4-x^5)
\sum_{n\ge1}A_n x^n \\
&=
\sum_{n\ge1}A_n x^n
-3\sum_{n\ge1}A_n x^{n+1}
+3\sum_{n\ge1}A_n x^{n+2} \\
&\quad
-2\sum_{n\ge1}A_n x^{n+3}
+2\sum_{n\ge1}A_n x^{n+4}
-\sum_{n\ge1}A_n x^{n+5}.
\end{align*}

Reindexing each sum to express everything in terms of $x^n$, we obtain
\[
Q(x)A(x)
=
\sum_{n\ge1}
\bigl(
A_n
-3A_{n-1}
+3A_{n-2}
-2A_{n-3}
+2A_{n-4}
-A_{n-5}
\bigr)x^n,
\]
where we interpret $A_k=0$ for $k\le0$.

Since $P(x)$ has degree $6$, we have $[x^n]P(x) = 0$ for all $n \geq 7$. Therefore, extracting the coefficient of $x^n$ from $Q(x)A(x) = P(x)$ gives, for all $n \geq 7$,
\[
A_n - 3A_{n-1} + 3A_{n-2} - 2A_{n-3} + 2A_{n-4} - A_{n-5} = 0,
\]
or equivalently,

\begin{equation}
\label{eq:rec}
A_n = 3A_{n-1} - 3A_{n-2} + 2A_{n-3} - 2A_{n-4} + A_{n-5}.
\end{equation}

The recurrence \eqref{eq:rec} holds for all $n \geq 7$; 
the initial values $A_1$ through $A_6$, verified by direct 
enumeration, are
\[
A_1 = 1,\quad A_2 = 2,\quad A_3 = 5,\quad
A_4 = 8,\quad A_5 = 12,\quad A_6 = 18.
\]

\subsection{Asymptotic Growth}

The singularities of $A(x)$ arise from the poles at $x=1$ and from the roots of $1-x-x^3$.
Let $\rho \in (0,1)$ be the unique positive real root of
\[
1-x-x^3=0.
\]
Since $\rho<1$, it is the dominant singularity (i.e.\ the pole of smallest modulus), and it is simple.

Write
\[
A(x)=\frac{P(x)}{Q(x)}, \qquad Q(x)=(1-x)^2(1-x-x^3).
\]
Since $\rho$ is a simple root of $1-x-x^3$, it is also a simple pole of $A(x)$.

Near $x=\rho$, we have the local expansion
\[
A(x) \sim \frac{P(\rho)}{Q'(\rho)} \cdot \frac{1}{x-\rho}.
\]
Using
\[
\frac{1}{x-\rho} = -\frac{1}{\rho}\cdot\frac{1}{1-x/\rho},
\]
it follows that
\[
A(x) \sim -\frac{P(\rho)}{\rho\,Q'(\rho)} \cdot \frac{1}{1-x/\rho}.
\]
Extracting coefficients gives
\[
A_n \sim C\,\rho^{-n},
\qquad
C = -\frac{P(\rho)}{\rho\,Q'(\rho)}.
\]

We now compute $P(\rho)$ and $Q'(\rho)$. From the expression for $A(x)$,
\[
P(x) = -x^6 + x^5 - 3x^4 + 2x^3 - x^2 + x.
\]
Using the relation $\rho^3 = 1-\rho$, we compute
\[
\rho^4 = \rho(1-\rho), \qquad
\rho^5 = \rho^2(1-\rho), \qquad
\rho^6 = (1-\rho)^2.
\]
Substituting into $P(\rho)$,
\begin{align*}
P(\rho)
&= - (1-\rho)^2 + \rho^2(1-\rho) - 3\rho(1-\rho) + 2(1-\rho) - \rho^2 + \rho \\
&= \rho(2\rho - 1).
\end{align*}

Next,
\[
Q(x)=(1-x)^2(1-x-x^3),
\]
so
\[
Q'(x)=2(1-x)(-1)(1-x-x^3) + (1-x)^2(-1-3x^2).
\]
Since $1-\rho-\rho^3=0$, the first term vanishes at $x=\rho$, giving
\[
Q'(\rho)=(1-\rho)^2(-1-3\rho^2).
\]

Therefore
\[
C = \frac{2\rho-1}{(1-\rho)^2(1+3\rho^2)}.
\]

Numerically,
\[
\rho \approx 0.6823278038,
\qquad
\alpha := \rho^{-1} \approx 1.4655712319,
\qquad
C \approx 1.5076770639.
\]

In particular, the exponential growth constant is $\alpha$, which is the unique real root $>1$ of
\[
\alpha^3 = \alpha^2 + 1.
\]

\section{General Structural Conjectures}
\label{sec:conjectures}
The cases $m=1$ and $m=2$ strongly suggest that the interaction between $132$-avoidance and the adjacency constraint produces a systematic structural collapse for fixed $m$.

In this section we formulate conjectures describing this phenomenon.
\subsection{Finite-State Structure}

\medskip

The structural decomposition obtained for $m=2$ shows that the class $A_n^{(2)}$ collapses to finitely many local configurations determined by the initial entries of the permutation. Each configuration evolves according to a fixed set of transition rules, yielding a finite-state description of the class and consequently a rational generating function.

This motivates the following conjecture.

\begin{conjecture}[Finite-state structure]\label{conj:finite-state}
For every fixed integer $m \ge 1$, the class
\[
A_n^{(m)} =
\{\pi \in S_n : \pi \text{ avoids } 132 \text{ and }
|\pi_{i+1}-\pi_i| \le m \}
\]
admits a finite structural decomposition into a bounded collection of
states determined by the local configuration near the beginning of the
permutation. Transitions between these states depend only on this local
information.
\end{conjecture}

\medskip

If Conjecture~\ref{conj:finite-state} holds, then the enumeration of the
class may be described by a finite transition system. In particular,
the sequence $\{A_n^{(m)}\}$ satisfies a linear recurrence with constant
coefficients, and the generating function
\[
A^{(m)}(x) = \sum_{n\ge0} A_n^{(m)} x^n
\]
is rational.

\subsection{Growth Constants and Limiting Behavior}

The extremal cases illustrate two regimes:

\begin{itemize}
    \item For $m=1$, one has $A_n^{(1)}=2$ for $n\ge2$,
    so the growth rate is $1$.
    
    \item For $m \ge n-1$, one recovers the Catalan numbers,
    whose exponential growth rate is $4$.
\end{itemize}

For $m=2$, we obtained
\[
A_n^{(2)} = \Theta(\alpha_2^n),
\qquad
\alpha_2 \approx 1.46557.
\]

Empirical data for $m=3$ suggests
\[
1 < \alpha_2 < \alpha_3 < 4.
\]

This leads naturally to the following conjecture.

\begin{conjecture}[Monotonic growth constants]\label{conj:monotonic-growth}
For fixed $m$, the limit
\[
\alpha_m = \lim_{n\to\infty} (A_n^{(m)})^{1/n}
\]
exists, and
\[
1 = \alpha_1 < \alpha_2 < \alpha_3 < \cdots < 4.
\]
Moreover,
\[
\lim_{m\to\infty} \alpha_m = 4.
\]
\end{conjecture}

\medskip

\noindent
\textbf{Heuristic justification.}
Increasing $m$ weakens the adjacency constraint,
thereby enlarging the permutation class. Thus $A_n^{(m)}$ is nondecreasing in $m$ for each fixed $n$, implying $\alpha_m$ is nondecreasing. As $m$ grows relative to $n$, the constraint becomes vacuous, and the Catalan growth rate $4$ is recovered.

\section{Open Problems and Future Directions}
\label{sec:future}
The case $m=2$ admits a complete structural and analytic solution. The case $m=3$ exhibits clear structural restrictions but does not collapse as cleanly. These observations suggest several natural directions for further study.

\subsection{Exact Solution for $m=3$}

The first open problem is to obtain a complete structural description
of the class
\[
A_n^{(3)} =
\left\{
\pi \in S_n :
\pi \text{ avoids } 132
\text{ and }
|\pi_{i+1}-\pi_i|\le 3
\right\}.
\]

The position of the maximum element has been restricted to
\[
\{1,2,3,n\},
\]
and preliminary analysis indicates that only finitely many local
boundary configurations arise. However, unlike the case $m=2$,
each placement of $n$ leads to multiple viable branches that do not immediately collapse into forced configurations.

An exact solution would require:

\begin{itemize}
    \item A complete classification of refined subclasses
    (e.g.\ configurations determined by the first few entries).
    \item A proof that only finitely many such boundary types occur.
    \item A derivation of the corresponding linear recurrence and
    rational generating function.
\end{itemize}

Establishing this would confirm the finite-state conjecture
in the first genuinely branching regime.

\subsection{Higher Values of $m$}

For general fixed $m$, the structural complexity appears to grow with $m$, but remains locally constrained. In particular, as shown in Theorem~\ref{thm:max-position-general-m}, the maximum element $n$ may occur only in positions
\[
\{1,2,\dots,m,n\}.
\]
This restriction reduces the Catalan decomposition to finitely many positional cases for each fixed $m$, suggesting that the class may admit a finite structural description.

Two major open questions remain:

\begin{enumerate}
    \item Does a finite-state decomposition exist for every fixed $m$?
    \item If so, can one describe an explicit recursive scheme analogous to the $m=2$ case?
\end{enumerate}

Even partial progress — for example, bounding the number of admissible boundary types as a function of $m$ — would clarify the global structure of the class.

Another natural direction is to determine how the order of the linear recurrence (if it exists) grows with $m$, e.g.\ whether it is linear, polynomial, or exponential.

\subsection{Connections to Other Catalan-Type Classes}

Classical $132$-avoiding permutations are counted by the Catalan numbers
and are in bijection with numerous Catalan objects, including
binary trees and Dyck paths.

It is natural to ask whether the adjacency-constrained classes
$A_n^{(m)}$ admit analogous interpretations.

Possible directions include:

\begin{itemize}
    \item Interpreting the adjacency constraint as a bounded-slope
    condition on Dyck paths.
    \item Relating the class to pattern-avoiding permutations
    with additional vincular or mesh restrictions.
    \item Investigating whether $A_n^{(m)}$ fits into a broader family
    of Lipschitz-constrained permutation classes.
\end{itemize}

Such connections may provide alternative structural proofs,
geometric interpretations of the finite-state phenomenon,
or new analytic tools for asymptotics.

\medskip

Overall, the interaction between global pattern avoidance and local
adjacency constraints appears to produce a new family of permutation
classes interpolating between rigid monotone structures and the
full Catalan regime. Understanding this interpolation more deeply
remains an intriguing problem.

\clearpage

\bibliographystyle{plain}
\bibliography{references}

\end{document}